\documentclass[12pt,letterpaper]{article}
\usepackage{amsthm,amssymb}
\usepackage{amsfonts}
\makeatletter
\usepackage[dvips]{color}
\usepackage{colordvi,multicol}

\newtheorem{thm}{\textbf Theorem}[section]
\newtheorem{lem}[thm]{\textbf Lemma}
\newtheorem{rem}[thm]{\textbf Remark}
\newtheorem{cor}[thm]{\textbf Corollary}

\newtheorem{prop}[thm]{\textbf Proposition}
\newtheorem{defin}[thm]{\textbf Definition}

\newcommand{\be}{\begin{eqnarray*}}
\newcommand{\ee}{\end{eqnarray*}}

\begin{document}

\title{\bf Stochastic Calculus for Markov Processes Associated with Non-symmetric Dirichlet Forms}
  \author{Chuan-Zhong Chen\\
Department of Mathematics\\
Hainan Normal University\\
Haikou, 571158, China\\
ccz0082@yahoo.com.cn\\
\\
Li Ma\\
Department of Mathematics\\
Hainan Normal University\\
Haikou, 571158, China\\
mary\hskip -0.09cm\_\hskip 0.05cm henan@yahoo.com.cn\\
\\
Wei Sun
            \\ Department of Mathematics and Statistics\\
             Concordia University\\
             Montreal, H3G 1M8, Canada\\
             wsun@mathstat.concordia.ca}
   \date{}
\maketitle

\begin{abstract}
\noindent Nakao's stochastic integrals for continuous additive functionals of zero energy are extended from the symmetric Dirichlet forms setting to the non-symmetric Dirichlet forms setting. It\^{o}'s formula in terms of the extended stochastic integrals is obtained.\vskip 0.5cm \noindent {\bf Keywords:} Non-symmetric Dirichlet form, Fukushima's decomposition, continuous additive functional of zero energy, stochastic integral, It\^{o}'s formula.
\end{abstract}
%%%%%%%%%%%%%%%%%%%%%%%%%%%%%%%%%%%%%%%%%%%%%%%%%%%%%%%%%%%%%%%%%%%%%%%
\section[short title]{Introduction}

Let $E$ be a
metrizable Lusin space (i.e., topologically isomorphic to a Borel subset of a
complete separable metric space) and $m$ be a $\sigma$-finite positive
measure on its Borel $\sigma$-algebra ${\cal B}(E)$. Suppose that $({\cal E},D({\cal
E}))$ is a quasi-regular (non-symmetric) Dirichlet form on $L^2(E;m)$ with
associated Markov process $((X_t)_{t\ge 0}, (P_x)_{x\in
E_{\Delta}})$ (we refer the reader to \cite{MR92} and \cite{Fu94} for notations and terminologies of this paper). If $u\in
D({\cal E})$, then there exist unique martingale additive functional
(MAF in short) $M^{[u]}$ of finite energy and continuous
additive functional (CAF in short) $N^{[u]}$ of zero energy
such that
\begin{equation}\label{31}
A_t^{[u]}:=\tilde{u}(X_t)-\tilde{u}(X_0)=M^{[u]}_t+N^{[u]}_t,
\end{equation}
where $\tilde{u}$ is an ${\cal E}$-quasi-continuous $m$-version of
$u$ and the energy of an AF $A:=(A_t)_{t\ge 0}$ is defined to be
$$
e(A):=\lim_{t\rightarrow 0}\frac{1}{2t}E_m[A^2_t]
$$
whenever the limit exists in $[0,\infty]$ (cf. \cite[Theorem VI.2.5]{MR92}). To simplify notation, in the sequel we take $u$ to be its ${\cal E}$-quasi-continuous $m$-version whenever such a version exists.

The aim of this paper is to define an integral $\int_0^tv(X_s)dA_s^{[u]}$ for a suitable class of functions $v$ on $E$. To this end, we need to define an integral $\int_0^tv(X_s)dN_s^{[u]}$. Note that in general $(N^{[u]}_t)_{t\ge 0}$ is not of bounded variation (cf. \cite[Example 5.5.2]{Fu94} for an example) and hence $(A^{[u]}_t)_{t\ge 0}$ is not a semi-martingale. In \cite{N}, Nakao defined stochastic integrals for CAFs and established the corresponding It\^{o}'s formula in the framework of symmetric Dirichlet forms. To define the stochastic integrals, he essentially used the assumption of symmetry and the formulae of Beurling-Deny and LeJan. Although both the Beurling-Deny decomposition and LeJan's transformation rule can be extended to the non-symmetric Dirichlet forms setting, the SPV integrability in the Beurling-Deny decomposition (cf. \cite{HC06} and \cite[Section 4]{MS11}) and the unavailability of the mutual energy measure for the co-symmetric diffusion part of $({\cal E},D({\cal
E}))$ (cf. \cite{HS10} and \cite[Section 4]{MS11}) make it difficult to directly extend Nakao's stochastic integrals to the non-symmetric Dirichlet forms setting.

In the next section, we will define the integral $\int_0^tv(X_s)dN_s^{[u]}$ without using the Beurling-Deny formula. In Section 3, we will establish It\^{o}'s formula in terms of our extended stochastic integrals.

%%%%%%%%%%%%%%%%%%%%%%%%%%%%%%%%%%%%%%%%%%%%%%%%%%%%%%%%%%%%%%%%%%%%%%%%%%%%%%%%%%%%%%%%%%%%%%%%%%
\section [short title]{Definition of stochastic integrals for CAFs}\label{Sec:Fukushima}
\setcounter{equation}{0}

We denote by $(\tilde{\cal E},D({\cal
E}))$ the symmetric part of $({\cal E},D({\cal
E}))$:
$$
\tilde{\cal E}(f,g):=\frac{1}{2}({\cal E}(f,g)+{\cal E}(g,f)),\ \ f,g\in D({\cal
E}).
$$
For $\alpha\ge 0$, we set ${\cal E}_{\alpha}(f,g):={\cal E}(f,g)+\alpha(f,g),\ f,g\in D({\cal
E})$, where $(\cdot,\cdot)$ is the inner product of $L^2(E;m)$.
\begin{lem}
Let $f\in D({\cal E})$. Then there exist unique $f^*\in D({\cal E})$ and $f^{\triangle}\in D({\cal E})$ such that for any $g\in D({\cal E})$,
\begin{equation}\label{1}
{\cal E}_1(f,g)=\tilde{\cal E}_1(f^*,g)
\end{equation}
and
\begin{equation}\label{2}
\tilde{\cal E}_1(f,g)={\cal E}_1(f^{\triangle},g).
\end{equation}
\end{lem}
\begin{proof} Let $f\in D({\cal E})$. Since $D({\cal E})$ is a Hilbert space w.r.t. the $\tilde{\cal E}^{1/2}_1$-norm, the assertion on $f^*$ is a direct consequence of the weak sector condition and the Riesz representation theorem. For the assertion on $f^{\triangle}$, the uniqueness of $f^{\triangle}$ is obvious. Below we will show that there exists an $f^{\triangle}\in D({\cal E})$ such that (\ref{2}) holds for any $g\in D({\cal E})$. To this end, we consider the map $Q: D({\cal E})\rightarrow D({\cal E})$, $Qw:=w^*$, $w\in D({\cal E})$. We will prove that $Q$ is surjective below. Once this done, one finds that there exists a $w\in D({\cal E})$ such that $f=w^*$. We define $f^{\triangle}:=w$ and then the proof is completed.

Let $w\in D({\cal E})$. Then, by (\ref{1}), we get
$$
{\cal E}_1(w,w)=\tilde{\cal E}_1(w^*,w)\le \tilde{\cal E}^{1/2}_1(w^*,w^*)\tilde{\cal E}^{1/2}_1(w,w),
$$
which implies that
\begin{equation}\label{3}
{\cal E}_1(w,w)\le {\cal E}_1(w^*,w^*).
\end{equation}
By (\ref{3}), one finds that $Q(D({\cal E}))$ is a closed subspace of $D({\cal E})$ w.r.t. the $\tilde{\cal E}^{1/2}_1$-norm. In fact, let $\{w^*_n\}$ be a sequence in $Q(D({\cal E}))$ which converges to $p\in D({\cal E})$ as $n\rightarrow\infty$. By (\ref{3}), there exists a $w\in D({\cal E})$ such that $w_n$ converges to $w$ in $D({\cal E})$ as $n\rightarrow\infty$. Then, for $g\in D({\cal E})$, we have
$$
{\cal E}_1(w,g)=\lim_{n\rightarrow\infty}{\cal E}_1(w_n,g)=\lim_{n\rightarrow\infty}\tilde{\cal E}_1(w^*_n,g)=\tilde{\cal E}_1(p,g),
$$
which implies that $p=w^*$ by (\ref{3}). Hence $Q(D({\cal E}))$ is a closed subspace of $D({\cal E})$.

Finally, we show that $Q(D({\cal E}))=D({\cal E})$. In fact, if $Q(D({\cal E}))\not=D({\cal E})$ then there exists a $q\in D({\cal E})$ such that $q\not=0$ and
$$
{\cal E}_1(h,q)=\tilde{\cal E}_1(h^*,q)=0,\ \ \forall h\in D({\cal E}).
$$
This implies that ${\cal E}_1(q,q)=0$, which is a contradiction. Therefore $Q(D({\cal E}))=D({\cal E})$.\end{proof}

Let $f,g\in D({\cal E})$. We use $\tilde{\mu}_{<f,g>}$ to denote the mutual energy measure of $f$ and $g$ w.r.t. the symmetric Dirichlet form $(\tilde{\cal E},D({\cal
E}))$. Suppose that $u\in D({\cal E})$ and $v\in D({\cal E})_b:=D({\cal E})\cap {\cal B}_b(E)$. It is easy to see that there exists a unique element in $D({\cal E})$, which is denoted by $\lambda(u,v)$, such that
\begin{equation}\label{14}
\frac{1}{2}\int_E vd\tilde{\mu}_{<h,u^*>}=\tilde{\cal E}_1(\lambda(u,v),h),\ \ \forall h\in D({\cal E}).
\end{equation}
\begin{thm}
Let $u\in D({\cal E})$ and $v\in D({\cal E})_b$. Then, for any $h\in D({\cal E})_b$,
\begin{equation}\label{11}
{\cal E}(u,hv)={\cal E}_1(\lambda(u,v)^{\triangle},h)+\frac{1}{2}\int_E hd\tilde{\mu}_{<v,u^*>}+\int_E (u^*-u)hvdm.
\end{equation}
\end{thm}
\begin{proof}
Let $f,g,w\in D({\cal E})_b$. Then, by \cite[Theorem 5.2.3]{Fu94} (and quasi-homeomorphism, see \cite{CMR}), we get
\begin{eqnarray*}
\int_E fd\tilde{\mu}_{<w,g>}&=&\tilde{\cal E}(w,fg)+\tilde{\cal E}(g,wf)-\tilde{\cal E}(wg,f),\\
\int_E wd\tilde{\mu}_{<f,g>}&=&\tilde{\cal E}(f,wg)+\tilde{\cal E}(g,wf)-\tilde{\cal E}(fg,w).
\end{eqnarray*}
Summing them up, we get
\begin{equation}\label{13}
\int_E fd\tilde{\mu}_{<w,g>}+\int_E wd\tilde{\mu}_{<f,g>}=2\tilde{\cal E}(g,wf).
\end{equation}
Further, by approximation, we can show that (\ref{13}) holds for any $f,w\in D({\cal E})_b$ and $g\in D({\cal E})$.

Let $u\in D({\cal E})$ and $v,h\in D({\cal E})_b$. By (\ref{1}), we get
\begin{equation}\label{12}
{\cal E}(u,hv)=\tilde{\cal E}_1(u^*,hv)-(u,hv)=\tilde{\cal E}(u^*,hv)+(u^*-u,hv).
\end{equation}
By (\ref{13}), we get
\begin{equation}\label{15}
\tilde{\cal E}(u^*,hv)=\frac{\int_E hd\tilde{\mu}_{<v,u^*>}+\int_E vd\tilde{\mu}_{<h,u^*>}}{2}.
\end{equation}
Therefore, (\ref{11}) holds by (\ref{12}), (\ref{15}), (\ref{14}) and (\ref{2}).
\end{proof}

Denote by $A_c^+$ the family of all positive CAFs (PCAFs in short) of $X$.  Define $$A_c^{+,f}:=\{A\in A_c^+\,|\,{\rm the\ smooth\ measure},\ \mu_A,\ {\rm cossreponding\ to}\ A\ {\rm is\ finite}\}$$
and
$$
{\cal N}^*_c:=\left\{N^{[u]}_t+\int_0^tf(X_s)ds+A^{(1)}_t-A^{(2)}_t\,|\,u\in D({\cal E}),f\in L^2(E;m)\ {\rm and}\ A^{(1)},A^{(2)}\in A_c^{+,f}\right\}.
$$
Similar to \cite[Theorem 2.2]{N}, we can prove the following lemma.
\begin{lem}\label{lem2}
If $C^{(1)}, C^{(2)}\in {\cal N}^*_c$ satisfying
$$
\lim_{t\downarrow 0}\frac{1}{t}E_{h\cdot m}[C^{(1)}_t]=\lim_{t\downarrow 0}\frac{1}{t}E_{h\cdot m}[C^{(2)}_t],\ \ \forall h\in D({\cal E})_b,
$$
then $C^{(1)}=C^{(2)}$.
\end{lem}
Let $u\in D({\cal E})$ and $v\in D({\cal E})_b$. Then $\tilde{\mu}_{<v,u^*>}$ is a signed smooth measure w.r.t. $(\tilde{\cal E},D({\cal
E}))$ and hence $({\cal E},D({\cal E}))$.  We use $G(u,v)$ to denote the unique element in  $A_c^+-A_c^+$ that is corresponding to $\tilde{\mu}_{<v,u^*>}$ under the Revuz correspondence between smooth measures of $({\cal E},D({\cal E}))$ and PCAFs of $X$ (cf. \cite[theorem VI.2.4]{MR92}). To simplify notation, we define
$$
\Gamma(u,v)_t:=N^{[\lambda(u,v)^{\triangle}]}_t-\int_0^t\lambda(u,v)^{\triangle}(X_s)ds,\ \ t\ge0.
$$

\begin{defin}\label{00} Let $u\in D({\cal E})$ and $v\in D({\cal E})_b$. We define for $t\ge 0$,
\begin{eqnarray*}
\int_0^tv(X_{s-})dN^{[u]}_s&:=&\int_0^tv(X_s)dN^{[u]}_s\\
&:=&\Gamma(u,v)_t-\frac{1}{2}G(u,v)_t-\int_0^t(u^*-u)v(X_s)ds.
\end{eqnarray*}
\end{defin}

\begin{rem}\label{rem1}

Let $u\in D({\cal E})$ and $v\in D({\cal E})_b$. Then one can check that $\int_0^tv(X_s)dN^{[u]}_s\in {\cal N}^*_c$. By Definition \ref{00}, (\ref{31}), \cite[Theorem 3.4]{f}, \cite[Theorem 5.8(iii)]{MMS} and (\ref{11}), we obtain that
\begin{equation}\label{w2}
\lim_{t\downarrow 0}\frac{1}{t}E_{h\cdot m}\left[\int_0^tv(X_s)dN^{[u]}_s\right]=-{\cal E}(u,hv).
\end{equation}
Therefore, by Lemma \ref{lem2}, $\int_0^tv(X_s)dN^{[u]}_s$ is the unique AF $(C_t)_{t\ge 0}$ in ${\cal N}^*_c$ that satisfies $\lim_{t\downarrow 0}\frac{1}{t}E_{h\cdot m}[C_t]=-{\cal E}(u,hv)$.
\end{rem}

Denote by $(L,D(L))$ the generator of $({\cal E},D({\cal
E}))$. Note that if $u\in D(L)$ then $dN^{[u]}_s=Lu(X_s)ds$. In this case, it is easy to see that for any $v,h\in D({\cal E})_b$,
$$
\lim_{t\downarrow 0}\frac{1}{t}E_{h\cdot m}\left[\int_0^tv(X_s)Lu(X_s)ds\right]=\int_EhvLudm=-{\cal E}(u,hv)
$$
(cf. \cite[Theorem 5.8(vi)]{MMS}). Hence our definition of the stochastic integrals $\int_0^tv(X_s)dN^{[u]}_s$ for $u\in D({\cal E})$ is an extension of the ordinary Lebesgue integrals $\int_0^tv(X_s)Lu(X_s)ds$ for $u\in D(L)$. More generally, we have the following proposition.

\begin{prop} Let $u\in D({\cal E})$ and $v\in D({\cal E})_b$. Suppose that there exist $A^{(1)},A^{(2)}\in A_c^+$ such that $N^{[u]}_t=A^{(1)}_t-A^{(2)}_t$ for $t<\zeta$, where $\zeta$ is the life time of $X$. Then
\begin{equation}\label{a1}
\int_0^tv(X_s)dN^{[u]}_s=\int_0^tv(X_s)d(A^{(1)}_s-A^{(2)}_s)\ {\rm for}\ t<\zeta.
\end{equation}
\end{prop}
\begin{proof}
By Definition \ref{00}, in order to prove (\ref{a1}) it suffices to prove that for $t<\zeta$,
\begin{eqnarray}\label{a2}
N^{[\lambda(u,v)^{\triangle}]}_t&=&\int_0^t\lambda(u,v)^{\triangle}(X_s)ds+\frac{1}{2}G(u,v)_t\nonumber\\
& &+\int_0^t(u^*-u)v(X_s)ds+\int_0^tv(X_s)d(A^{(1)}_s-A^{(2)}_s).
\end{eqnarray}
By quasi-homeomorphism, we assume without loss of generality that $({\cal E},D({\cal
E}))$ is a regular Dirichlet form. Denote by $B_t$ the right hand side of (\ref{a2}). Then $(B_t)_{t\ge 0}\in A_c^+-A_c^+$. By (\ref{11}), it is not hard to show that there exists a common compact nest $\{K_l\}$ such that
$$
{\cal E}(\lambda(u,v)^{\triangle},h)=-\langle I_{K_l}\cdot\mu_B,h\rangle,\ \ \forall h\in D({\cal E})_{K_l}\cap {\cal B}_b(E),\ l=1,2,\dots
$$
Therefore, (\ref{a2}) holds by the analog of \cite[Theorem 5.4.2]{Fu94} in the non-symmetric Dirichlet forms setting (cf. \cite[Theorem 5.2.7]{oshima}). The proof is complete.
\end{proof}

\begin{thm}\label{pb} Let $v\in D({\cal E})_b$ and $\{u_n\}_{n=0}^{\infty}\subset D({\cal E})$ satisfying $u_n$ converges to $u_0$ w.r.t. the $\tilde{\cal E}^{1/2}_1$-norm as $n\rightarrow\infty$. Then there exists a subsequence $\{n'\}$ such that for ${\cal E}{\textrm{-}q.e.}$ $x\in E$,
$$
P_x\left(\lim_{n'\rightarrow\infty}\int_0^tv(X_s)dN^{[u_{n'}]}_s=\int_0^tv(X_s)dN^{[u_0]}_s\ {\rm uniformly\ on\ any\ finite\ interval\ of}\ t\right)=1.                                                                                                          $$
\end{thm}
\begin{proof} By Definition \ref{00}, we have
\begin{eqnarray*}
\int_0^tv(X_s)dN^{[u_n]}_s&=&N^{[\lambda(u_n,v)^{\triangle}]}_t-\int_0^t\lambda(u_n,v)^{\triangle}(X_s)ds\\
& &-\frac{1}{2}G(u_n,v)_t-\int_0^t(u_n^*-u_n)v(X_s)ds,\ \ n=0,1,2,\dots
\end{eqnarray*}
For each term of the right hand side of the above equation, we will prove that  there exists a subsequence which converges uniformly on\ any\ finite\ interval\ of $t$.

(i) Term $N^{[\lambda(u_n,v)^{\triangle}]}_t$: Suppose $u\in D({\cal E})$ and $v\in D({\cal E})_b$. By (\ref{14}) and \cite[(5.6.2), (5.2.8), (5.2.3)]{Fu94}, we get
\begin{eqnarray*}
{\cal E}^{1/2}_1(\lambda(u,v),\lambda(u,v))&=&\sup_{h\in D({\cal E})_b,{\cal E}_1(h,h)=1}\frac{1}{2}\left|\int_E vd\tilde{\mu}_{<h,u^*>}\right|\nonumber\\
&\le&\|v\|_{\infty}{\cal E}^{1/2}_1(u^*,u^*).
\end{eqnarray*}
Note that (\ref{1}) and (\ref{2}) imply the continuity of the maps $f\rightarrow f^*$ and $f\rightarrow f^{\triangle}$. Then $\lambda(u_n,v)^{\triangle}$ converges to $\lambda(u_0,v)^{\triangle}$ w.r.t. the $\tilde{\cal E}^{1/2}_1$-norm as $n\rightarrow\infty$. Therefore the assertion on the term $N^{[\lambda(u_n,v)^{\triangle}]}_t$ is proved by the analog of \cite[Corollary 5.2.1(ii)]{Fu94} in the non-symmetric Dirichlet forms setting.

(ii) Term $\int_0^t\lambda(u_n,v)^{\triangle}(X_s)ds$: Denote by $S_0$ the family of all measures of finite energy integral. Suppose that $u\in D({\cal E})$ and $\nu\in S_0$. Denote by $(T_t)_{t\ge 0}$ the $L^2$-semigroup associated with $({\cal E},D({\cal E}))$. By \cite[Remark 3.3(ii)]{f}, the restriction of $(T_t)_{t\ge 0}$ on $D({\cal E})$ is a strongly continuous semigroup on $D({\cal E})$. Then there exist $c,M>0$ such that
$$
{\cal E}^{1/2}_1(T_tu,T_tu)\le ce^{Mt}{\cal E}^{1/2}_1(u,u),\ \ \forall u\in D({\cal E}), t\ge 0.
$$
Thus
\begin{eqnarray}\label{z2}
E_{\nu}\left[\sup_{0\le s\le t}\left|\int_0^su(X_r)dr\right|\right]&\le&\int_0^t\langle\nu,T_s|u|\rangle ds\nonumber\\
&=&\int_0^t{\cal E}_1(U_1\nu,T_s|u|) ds\nonumber\\
&\le&K{\cal E}^{1/2}_1(U_1\nu,U_1\nu)\int_0^t{\cal E}^{1/2}_1(T_s|u|,T_s|u|) ds\nonumber\\
&\le&Kce^{Mt}t{\cal E}^{1/2}_1(U_1\nu,U_1\nu){\cal E}^{1/2}_1(|u|,|u|)\nonumber\\
&\le&Kce^{Mt}t{\cal E}^{1/2}_1(U_1\nu,U_1\nu){\cal E}^{1/2}_1(u,u),
\end{eqnarray}
where $K>0$ is the continuity constant (cf. \cite[I.(2.3)]{MR92}).

By (i), $\lambda(u_n,v)^{\triangle}$ converges to $\lambda(u_0,v)^{\triangle}$ w.r.t. the $\tilde{\cal E}^{1/2}_1$-norm as $n\rightarrow\infty$. Then, by (\ref{z2}) and the same method as in the proof of \cite[Lemma 5.1.2]{Fu94} (cf. \cite[Theorem 2.3.8]{oshima}), the assertion on the term $\int_0^t\lambda(u_n,v)^{\triangle}(X_s)ds$ is proved.

(iii) Recall that, for $u\in D({\cal E})$ and $v\in D({\cal E})_b$, $G(u,v)$ denotes the unique element in  $A_c^+-A_c^+$ that is corresponding to $\tilde{\mu}_{<v,u^*>}$ under the Revuz correspondence between smooth measures of $({\cal E},D({\cal E}))$ and PCAFs of $X$. We use  $G^+(u,v)$ and $G^-(u,v)$ to denote the PACFs corresponding to $\tilde{\mu}^+_{<v,u^*>}$ and
$\tilde{\mu}^-_{<v,u^*>}$, respectively.

Define
$$
\hat{S}_{00}:=\{\mu\in S_0\,|\, \mu(E)<\infty,\|\hat{U}_1\mu\|_{\infty}<\infty\},
$$
where $\hat{U}_1$ denotes the 1-copotential. Let $\nu\in S_{00}$. Then, by \cite[Lemma 5.1.1]{oshima} and \cite[(5.6.2), (5.2.8), (5.2.3)]{Fu94}, we get
\begin{eqnarray*}
& &E_{\nu}\left[\sup_{0\le s\le t}|G(u_n,v)_s-G(u,v)_s|\right]\\
&&\ \ \ \ \ \ \ \ =E_{\nu}\left[\sup_{0\le s\le t}|G(u_n-u,v)_s|\right]\\
&&\ \ \ \ \ \ \ \ \le E_{\nu}[|G^+(u_n-u,v)_t|]+E_{\nu}[|G^-(u_n-u,v)_t|]\\
&&\ \ \ \ \ \ \ \ \le(1+t)\|\hat{U}^1\nu\|_{\infty}|\tilde{\mu}_{<v,(u_n-u)^*>}|(E)\\
&&\ \ \ \ \ \ \ \ \le2(1+t)\|\hat{U}^1\nu\|_{\infty}|{\cal E}^{1/2}_1(v,v){\cal E}^{1/2}_1((u_n-u)^*,(u_n-u)^*),
\end{eqnarray*}
which converges to 0 as $n\rightarrow\infty$. The proof is completed by the same method as in the proof of \cite[Lemma 5.1.2]{Fu94} (cf. \cite[Theorem 2.3.8]{oshima}).

(iv) Term $\int_0^t(u_n^*-u_n)v(X_s)ds$: The proof is similar to that of (ii) by noting that, for $u\in D({\cal E})$ and $v\in D({\cal E})_b$,
$$
E_{\nu}\left[\sup_{0\le s\le t}\left|\int_0^suv(X_r)dr\right|\right]\le \|v\|_{\infty}\int_0^tE_{\nu}[|u|(X_s)]ds=\|v\|_{\infty}\int_0^t\langle\nu,T_s|u|\rangle ds.
$$
\end{proof}

 %%%%%%%%%%%%%%%%%%%%%%%%%%%%%%%%%%%%%%%%%%%%%%%%%%%%%%%%%%%%%%%%%%%%%%%%%%%%%%%%%%%%%%%%%%%%%%%%%%%%%%%%%%%%%%%%%
\section {It\^{o}'s formula}\setcounter{equation}{0}
\label{sec:transform} In this section, we will establish It\^{o}'s formula in terms of the stochastic integrals introduced in \S2. We refer the reader to \cite{F1,F2,CFKZ,kuwae2010,K2} for other developments in this connection. By quasi-homeomorphism, we assume without loss of generality in this section that $({\cal E},D({\cal
E}))$ is a regular (non-symmetric) Dirichlet form and its associated Markov process $X$ is a Hunt process.

\begin{prop}\label{pa} Let $u,v\in D({\cal E})_b$. Then
\begin{equation}\label{s1}
\int_0^tv(X_s)dN^{[u]}_s+\int_0^tu(X_s)dN^{[v]}_s=N^{[uv]}_t-\langle M^{[u]},M^{[v]}\rangle_t,\ \ t\ge 0.
\end{equation}
\end{prop}
\begin{proof} Let $u,v\in D({\cal E})_b$. Then, for any $h\in D({\cal E})_b$, we obtain by (\ref{w2}) and \cite[Theorem 5.1.5]{oshima} that
\begin{eqnarray*}
& &\lim_{t\downarrow 0}\frac{1}{t}E_{h\cdot m}\left[\int_0^tv(X_s)dN^{[u]}_s+\int_0^tu(X_s)dN^{[v]}_s\right]\\
&&\ \ \ \ \ \ \ \ =-{\cal E}(u,hv)-{\cal E}(v,hu)\\
&&\ \ \ \ \ \ \ \ =-{\cal E}(uv,h)-\int_Ehd\mu_{\langle M^{[u]},M^{[v]}\rangle}\\
&&\ \ \ \ \ \ \ \ =\lim_{t\downarrow 0}\frac{1}{t}E_{h\cdot m}\left[N^{[uv]}_t-\langle M^{[u]},M^{[v]}\rangle_t\right].
\end{eqnarray*}
Since $\int_0^tv(X_s)dN^{[u]}_s+\int_0^tu(X_s)dN^{[v]}_s\in {\cal N}^*_c$ and $N^{[uv]}_t-\langle M^{[u]},M^{[v]}\rangle_t\in {\cal N}^*_c$, (\ref{s1}) holds by Lemma \ref{lem2}.
\end{proof}

By Theorem \ref{pb}, Proposition \ref{pa} and \cite[Corollary I.4.15]{MR92}, we can prove the following corollary.

\begin{cor}\label{cor1} Let $u\in D({\cal E})_b$ and $\{v_n\}_{n=0}^{\infty}\subset D({\cal E})_b$ satisfying $v_n$ converges to $v_0$ w.r.t. the $\|\cdot\|_{\infty}$-norm and the $\tilde{\cal E}^{1/2}_1$-norm as $n\rightarrow\infty$. Then there exists a subsequence $\{n'\}$ such that for ${\cal E}{\textrm{-}q.e.}$ $x\in E$,
$$
P_x\left(\lim_{n'\rightarrow\infty}\int_0^tv_{n'}(X_s)dN^{[u]}_s=\int_0^tv_0(X_s)dN^{[u]}_s\ {\rm uniformly\ on\ any\ finite\ interval\ of}\ t\right)=1.                                                                                                          $$
\end{cor}

Let $(N,H)$ be a L\'evy system for $X$ (cf. \cite[Definition A.3.7]{Fu94}) and $\nu$ be the Revuz measure of the PCAF $H$. Similar to \cite[Theorem 5.3.1]{Fu94} (cf. also \cite[Chapter 5]{oshima}), we can show that the jumping measure $J$ and the killing measure $K$ of $X$ are given by
\begin{eqnarray}\label{sd10}
J(dx,dy)=\frac{1}{2}N(y,dx)\nu(dy),\ \ K(dx) = N(x,\Delta)\nu(dx).
\end{eqnarray}
Let $u\in D({\cal E})$. Then $M^{[u]}$ can be decomposed as
\begin{eqnarray}\label{sd0}
M^{[u]}_t=M^{[u],c}_t+M^{[u],j}_t+M^{[u],k}_t,\ \ t\ge 0,\ P_x{\textrm{-}a.s.}\ {\rm for}\ {\cal E}{\textrm{-}q.e.}\ x\in E,
\end{eqnarray}
where $M^{[u],c}_t$ is the continuous part of $M^{[u]}_t$ and
\begin{eqnarray*}
M^{[u],j}_t&=&\lim_{l\rightarrow\infty}\left\{\sum_{0<s\le t}(u(X_s)-u(X_{s-}))1_{\{|u(X_s)-u(X_{s-})|>1/l\}}1_{\{s<\zeta\}}\right.\\
& &\left.-\int_0^t\left(\int_{\{y\in E\,|\,|u(y)-u(X_s)|>1/l\}}(u(y)-u(X_s))N(X_s,dy)\right)dH_s\right\},\\
M^{[u],k}_t&=&\int_0^tu(X_s)N(X_s,\Delta)dH_s-u(X_{\zeta-})1_{\{t\ge \zeta\}}
\end{eqnarray*}
are the jumping and killing parts of $M^{[u]}$, respectively. The limit in the expression for $M^{[u],j}_t$ is in the sense of convergence w.r.t. the energy $e$ and of convergence in probability under $P_x$ for ${\cal E}{\textrm{-}q.e.}$ $x\in E$ for each fixed $t>0$ (cf. \cite[Theorem A.3.9 and page 341]{Fu94}). Define $\Delta u(X_t):= u(X_t)-u(X_{t-})$ for $t>0$.

\begin{thm}\label{thm1} Let $u,v\in D({\cal E})_b$. Then,
\begin{eqnarray}\label{sd1}
& &uv(X_t)-uv(X_0)=\int_0^tv(X_{s-})du(X_s)+\int_0^tu(X_{s-})dv(X_s)+\langle M^{[u],c},M^{[v],c}\rangle\nonumber\\
& &\ \ \ \ +\sum_{0<s\le t}[\Delta uv(X_s)-v(X_{s-})\Delta u(X_s)-u(X_{s-})\Delta v(X_s)].
\end{eqnarray}
\end{thm}
\begin{proof} By (\ref{31}), (\ref{s1}), (\ref{sd0}) and \cite[Theorem 5.4]{Kim}, we get
\begin{eqnarray*}
uv(X_t)-uv(X_0)&=&M^{[uv]}_t+N^{[uv]}_t\\
&=&M^{[uv],c}_t+M^{[uv],j}_t+M^{[uv],k}_t\\
& &+\int_0^tv(X_{s-})dN^{[u]}_s+\int_0^tu(X_{s-})dN^{[v]}_s+\langle M^{[u]},M^{[v]}\rangle_t\\
&=&\int_0^tv(X_{s-})dM^{[u],c}_s+\int_0^tu(X_{s-})dM^{[v],c}_s+M^{[uv],j}_t+M^{[uv],k}_t\\
& &+\int_0^tv(X_{s-})dN^{[u]}_s+\int_0^tu(X_{s-})dN^{[v]}_s+\langle M^{[u]},M^{[v]}\rangle_t.
\end{eqnarray*}
Then, in order to prove (\ref{sd1}), it suffices to prove that
\begin{eqnarray}\label{sd3}
& &M^{[uv],j}_t=\int_0^tv(X_{s-})dM^{[u],j}_s+\int_0^tu(X_{s-})dM^{[v],j}_s-\langle M^{[u],j},M^{[v],j}\rangle_t\nonumber\\
& &\ \ \ \ +\sum_{0<s\le t<\zeta}[\Delta uv(X_s)-v(X_{s-})\Delta u(X_s)-u(X_{s-})\Delta v(X_s)]\ \
\end{eqnarray}
and
\begin{eqnarray}\label{sd4}
M^{[uv],k}_t&=&\int_0^tv(X_{s-})dM^{[u],k}_s+\int_0^tu(X_{s-})dM^{[v],k}_s-\langle M^{[u],k},M^{[v],k}\rangle_t\nonumber\\
& &+uv(X_{\zeta-})1_{\{t\ge \zeta\}}.
\end{eqnarray}

Below we will only prove (\ref{sd3}). The proof of (\ref{sd4}) is similar so we omit it.

Let $\rho$ be the metric of $E$ and $\{K_l\}$ be an increasing sequence of compact subsets of $E$ satisfying $\lim_{l\rightarrow\infty}K_l=E$. We define
\begin{eqnarray*}
D_l&:=&\{(x,y)\in K_l\times K_l\,|\,\rho(x,y)\ge\frac{1}{l}\},\\
V^{u,l}_t&:=&\sum_{0<s\le t}\Delta u(X_s)1_{\{(X_s,X_{s-})\in D_l\}},\\
M^{u,j,l}_t&:=&V^{u,l}_t-(V^{u,l})^p_t.
\end{eqnarray*}
Hereafter $A^p$ denotes the dual predictable projection of $A$. Define
\begin{eqnarray*}
J^l_t&:=&\int_0^tv(X_{s-})dM^{u,j,l}_s+\int_0^tu(X_{s-})dM^{v,j,l}_s-\langle M^{u,j,l},M^{v,j,l}\rangle_t\\
& &\ \ \ \ +\sum_{0<s\le t<\zeta}[\Delta uv(X_s)-v(X_{s-})\Delta u(X_s)-u(X_{s-})\Delta v(X_s)]1_{\{(X_s,X_{s-})\in D_l\}}.
\end{eqnarray*}
Then
\begin{eqnarray*}
J^l_t-M^{uv,j,l}_t&=&-\int_0^tv(X_{s-})d(V^{u,l})^p_s-\int_0^tu(X_{s-})d(V^{v,l})^p_s\\
& &-\langle M^{u,j,l},M^{v,j,l}\rangle_t+(V^{uv,l})^p_t,
\end{eqnarray*}
which implies that $J^l_t-M^{uv,j,l}\in {\cal N}^*_c$.

Further, we obtain by (\ref{sd10}) that for any $h\in D({\cal E})_b$,
\begin{eqnarray*}
\lim_{t\downarrow 0}\frac{1}{t}E_{h\cdot m}[J^l_t-M^{uv,j,l}_t]&=&-2\int_{D_l}h(x)v(y)(u(x)-u(y))J(dy,dx)\\
& &-2\int_{D_l}h(x)u(y)(v(x)-v(y))J(dy,dx)\\
& &-2\int_{D_l}h(x)(u(x)-u(y))(v(x)-v(y))J(dy,dx)\\
& &+2\int_{D_l}h(x)(uv(x)-uv(y))J(dy,dx)\\
&=&0.
\end{eqnarray*}
Thus $J^l_t=M^{uv,j,l}_t$ by Remark \ref{rem1}. Denote by $J_t$ the right hand side of (\ref{sd3}). Since $\{K_l\}$ converges to $E$, there exists a subsequence $\{l'\}$ satisfying for ${\cal E}{\textrm{-}q.e.}$ $x\in E$, $P_x(J^{l'}_t\ {\rm converges\ to}\ J_t\ {\rm on\ each\ finite\ interval\ of}\ [0,\infty))=1$ and $P_x(M^{uv,j,l'}_t\ {\rm converges\ to}\ M^{[uv],j}_t\ {\rm on\ each\ finite\ interval\ of}\ [0,\infty))=1$. Therefore (\ref{sd3}) holds.
\end{proof}

\begin{thm} Let $\Phi\in C^2(\mathbb{R}^n)$ and $u_1,\dots,u_n\in D({\cal E})_b$. Then,
\begin{eqnarray}\label{ss1}
& &A_t^{[\Phi(u)]}=\sum_{i=1}^n\int_0^t\Phi_i(u(X_{s-}))dA^{[u_i]}_s+\frac{1}{2}\sum_{i,j=1}^n\Phi_{ij}(u(X_s))d\langle M^{[u_i],c},M^{[u_j],c}\rangle\nonumber\\
& &\ \ \ \ +\sum_{0<s\le t}\left[\Delta \Phi(u(X_s))-\sum_{i=1}^n\Phi_i(u(X_{s-}))\Delta u_i(X_s)\right],
\end{eqnarray}
where
$$
\Phi_i(x)=\frac{\partial\Phi}{\partial x_i}(x),\ \ \Phi_{ij}(x)=\frac{\partial^2\Phi}{\partial x_i\partial x_j}(x),\ \ i,j=1,\dots,n,
$$
and $u=(u_1,\dots,u_n)$.
\end{thm}
\begin{proof} By Theorem \ref{thm1} and induction, we know that (\ref{ss1}) holds for any polynomial $\Phi$. Denote by $V$ a finite cube containing the range of $u$. For a general $\Phi\in C^2(\mathbb{R}^n)$ vanishing at the origin, we choose a sequence of polynomials $\{\Phi_r\}$ vanishing at the origin such that $\Phi_r\rightarrow\Phi$, $(\Phi_r)_i\rightarrow\Phi_i$, $(\Phi_r)_{ij}\rightarrow\Phi_{ij}$, $1\le i,j\le n$, uniformly on $V$ as $r\rightarrow\infty$. By virtue of the approximation sequence $\{\Phi_r\}$, \cite[(3.2.27)]{Fu94} and Corollary \ref{cor1}, we find that (\ref{sd1}) holds for $\Phi$.
\end{proof}

%%%%%%%%%%%%%%%%%%%%%%%%%%%%%%%%%%%%%%%%%%%%%%%%%%%%%%%%%%%%%%%%%%%%%%%%%%%%%%%%%%%%%%%%%%%%%%%%%%%%%%%%
\bigskip

{ \noindent {\bf\large Acknowledgments} \vskip 0.1cm  \noindent We are
grateful to the support of NSFC (Grant No. 10961012) and
NSERC (Grant No. 311945-2008).}
\vskip 0.5cm
{ \noindent {\bf\large Note:} After this paper was submitted to arXiv.org, Dr. Alexander Walsh Zuniga emailed us and told us that he read our paper in arXiv and found that some results given in this paper were the same as the corresponding results given in Chapter 5 of his completed PhD thesis. Both Dr. Zuniga and the authors of this paper have agreed that the works of the two groups are independent and the methods are different. We thank Dr. Zuniga very much for letting us know his thesis, which can be found at
http://tel.archives-ouvertes.fr/tel-00627558\hskip -0.03cm\_\hskip 0cm v1/ .

%%%%%%%%%%%%%%%%%%%%%%%%%%%%%%%%%%%%%%%%%%%%%%%%%%%%%%%%%%%%%%%%%%%%%%%%%%%%%%%%%%%%%%%%%%%%%%%%%%%%%%

\end{document}